\documentclass[11pt]{revtex4}

\usepackage{graphicx,epsfig,amssymb,amsbsy,amsmath,amsthm}
\newcommand{\bb}[1]{\left({#1}\right)}					
\newcommand{\sq}[1]{\left[#1\right]}						
\newcommand{\cc}[1]{\left\{#1\right\}}					
\newcommand{\op}[1]{\mathcal{#1}}
\newcommand{\ord}[1]{{\sf O}\bb{#1}}					

\newcommand{\y}{\mathbf{y}}

\newcommand{\la}{\lambda}
\newcommand{\sgn}{\mathrm{sign}}
\newcommand{\R}{\ensuremath{\mathbb{R}}}
\newcommand{\fn}{K}
\newcommand{\ov}{\overline}
\newcommand{\de}{\delta}
\def\p{\partial}
\newcommand{\CO}{\ensuremath{\mathcal{O}}}
\newcommand{\CK}{\ensuremath{\mathcal{K}}}

\newtheorem{theorem}{Theorem}

\newtheorem{corollary}[theorem]{Corollary}

\newtheorem{example}{Example}

\newcommand{\man}{R}
\newcommand{\strip}{{\op I}}

\def\e{\varepsilon}

\begin{document}

\author{Douglas D. Novaes$^2$ and Mike R. Jeffrey$^1$}

\address{$^1$ Department of Engineering Mathematics,
University of Bristol, Merchant Venturer's Building, Bristol BS8 1UB, United Kingdom, email: mike.jeffrey@bristol.ac.uk}

\address{$^2$ Departamento de Matem\'{a}tica, Universidade Estadual de Campinas, Rua S\'{e}rgio Baruque de Holanda, 651, Cidade Universit\'{a}ria Zeferino Vaz, 13083--859, Campinas, SP,
Brazil, email: ddnovaes@ime.unicamp.br}

\title{Regularization of hidden dynamics in piecewise smooth flows}

\begin{abstract}
This paper studies the equivalence between differentiable and non-differentiable dynamics in $\mathbb R^n$. 
Filippov's theory of discontinuous differential equations allows us to find flow solutions of dynamical systems whose vector fields undergo switches at thresholds in phase space. The canonical {\it convex combination} at the discontinuity is only the linear part of a {\it nonlinear combination} that more fully explores Filippov's most general problem: the differential inclusion. Here we show how recent work relating discontinuous systems to singular limits of continuous (or {\it regularized}) systems extends to nonlinear combinations. We show that if sliding occurs in a discontinuous systems, there exists a differentiable slow-fast system with equivalent slow invariant dynamics. We also show the corresponding result for the {\it pinching} method, a converse to regularization which approximates a smooth system by a discontinuous one. 
\end{abstract}



\maketitle

Consider an ordinary differential equation in $x\in\mathbb R^n$ with a discontinuous righthand side, 
\begin{equation}\label{s1}
\dot x=\left\{\begin{array}{l}
f^+(x)\quad \textrm{if}\quad h(x)>0,\\
f^-(x)\quad \textrm{if}\quad h(x)<0,
\end{array}\right.
\end{equation}
where $f^+$ and $f^-$ are continuous vector fields, and $h$ is a differentiable scalar function whose gradient $\nabla h$ is well-defined and non-vanishing everywhere. Throughout this paper we consider an open region $x\in D$ in which (\ref{s1}) holds. The set $\Sigma=\{x\in\op D:\,h(x)=0\}$ is called the {\it switching manifold}, and the regions either side of it are denoted as $\op R^\pm=\cc{x\in\op D:\;h(x)\gtrless0}$. 

The term `hidden dynamics' refers to what happens on $\Sigma$, specifically to behaviours governed by terms that disappear in $\op R^\pm$ (hence they are `hidden' in (\ref{s1})), and which go beyond Filippov's standard theory \cite{F}. The theory of Filippov relies heavily on two alternatives for extending (\ref{s1}) across $h=0$. The first is a differential inclusion
\begin{equation}\label{inc}
\dot x\in\mathcal F(x)\qquad {\rm s.t.}\quad f^+(x),f^-(x)\in\mathcal F(x)\;
\end{equation}
which is very general because $\op F$ is any set that contains $f^\pm$ ($\op F$ is usually assumed to be convex  to provide certain restrictions on sequences of solutions \cite{F}, but this does not prevent $\op F$ being arbitrarily large). The second alternative is a smaller set, the convex hull of $f^+$ and $f^-$,
\begin{equation}\label{s2c}
\begin{array}{l}
\dot x=Z(x;\lambda):=\dfrac{1+\la}{2}f^+(x)+\dfrac{1-\la}{2}f^-(x),\quad
\lambda\in\left\{\begin{array}{lll}{\rm sign}(h(x))&\rm if&h(x)\neq0\;,\\\sq{-1,+1}&\rm if&h(x)=0\;,\end{array}\right.
\end{array}
\end{equation}
which is very restrictive in the sense that it selects only values of (\ref{inc}) that are linear combinations of $f^\pm$. Examples of the set $\op F$ and hull $\cc{Z(x;\lambda):\lambda\in[-1,+1]}$ will be illustrated in Example \ref{eg:fil} below, along with a third alternative that unties them. 

We will refer to the transition as $h$ changes sign in (\ref{s2c}) as {\it linear switching} (implying linear dependence with respect to $\lambda$). In Filippov's theory, one seeks values of $\dot x$ in the sets (\ref{inc}) or (\ref{s2c}) that result in continuous (though typically non-differentiable) flows at $\Sigma$. In many situations of interest, the flow obtained from (\ref{s2c}) is unique (making possible, for example, substantial classifications of singularities and bifurcations for such systems \cite{F,T,bc08}).  

The problem highlighted in \cite{J} was that between the set-valued flow of (\ref{inc}) and the piecewise-smooth flow of (\ref{s2c}), a vast expanse of non-equivalent but no less valid dynamical systems can be considered. All that is lacking is a way to express them explicitly. This is provided quite simply by permitting nonlinear dependence on the transition parameter $\lambda$, in the form
\begin{equation}\label{s2}
\dot x=f(x;\la):=\dfrac{1+\la}{2}f^+(x)+\dfrac{1-\la}{2}f^-(x)+G(x;\la),
\end{equation}
where
\begin{equation}\label{s2a}
 h(x)G(x;\la)=0\;,\qquad \lambda\in\left\{\begin{array}{lll}{\rm sign}(h(x))&\rm if&h(x)\neq0\;,\\\sq{-1,+1}&\rm if&h(x)=0\;,\end{array}\right.
\end{equation}
with $G$ some continuous vector field that is nonlinear in $\lambda$. An example of the set generated by $\cc{f(x;\lambda):\lambda\in[-1,+1]}$ is given in Example \ref{eg:fil} below. We shall refer to (\ref{s2}) as the nonlinear combination, and the transition it undergoes as $h$ changes sign as {\it nonlinear switching}. (Moreover the term `nonlinear' throughout this paper will refer to nonlinear dependence on $\lambda$ via the function $G$). 

\begin{example}\label{eg:fil}
Consider in coordinates $x=(x_1,x_2)$ the piecewise constant system (\ref{s1}) with vector fields
$f^+=(1,1)$, $f^-=(1,-2)$, and $G(\lambda)=(\lambda^2-1)(2,0)$, with $h(x)=x_1$. 
In Figure \ref{fig:fil} we illustrate a convex set $\op F$ satisfying (\ref{inc}), the linear combination $Z(x;\lambda)$ defined in (\ref{s2c}), and the nonlinear combination from (\ref{s2}), represented by the shaded region, dashed line, and dotted curve, respectively. By choosing different forms of $G$ (subject to $hG=0$) we can choose different curves $\cc{f(x;\lambda):\lambda\in[-1,+1]}$ which explore different subsets of $\op F$. 
\begin{figure}[h!]\centering\includegraphics[width=0.5\textwidth]{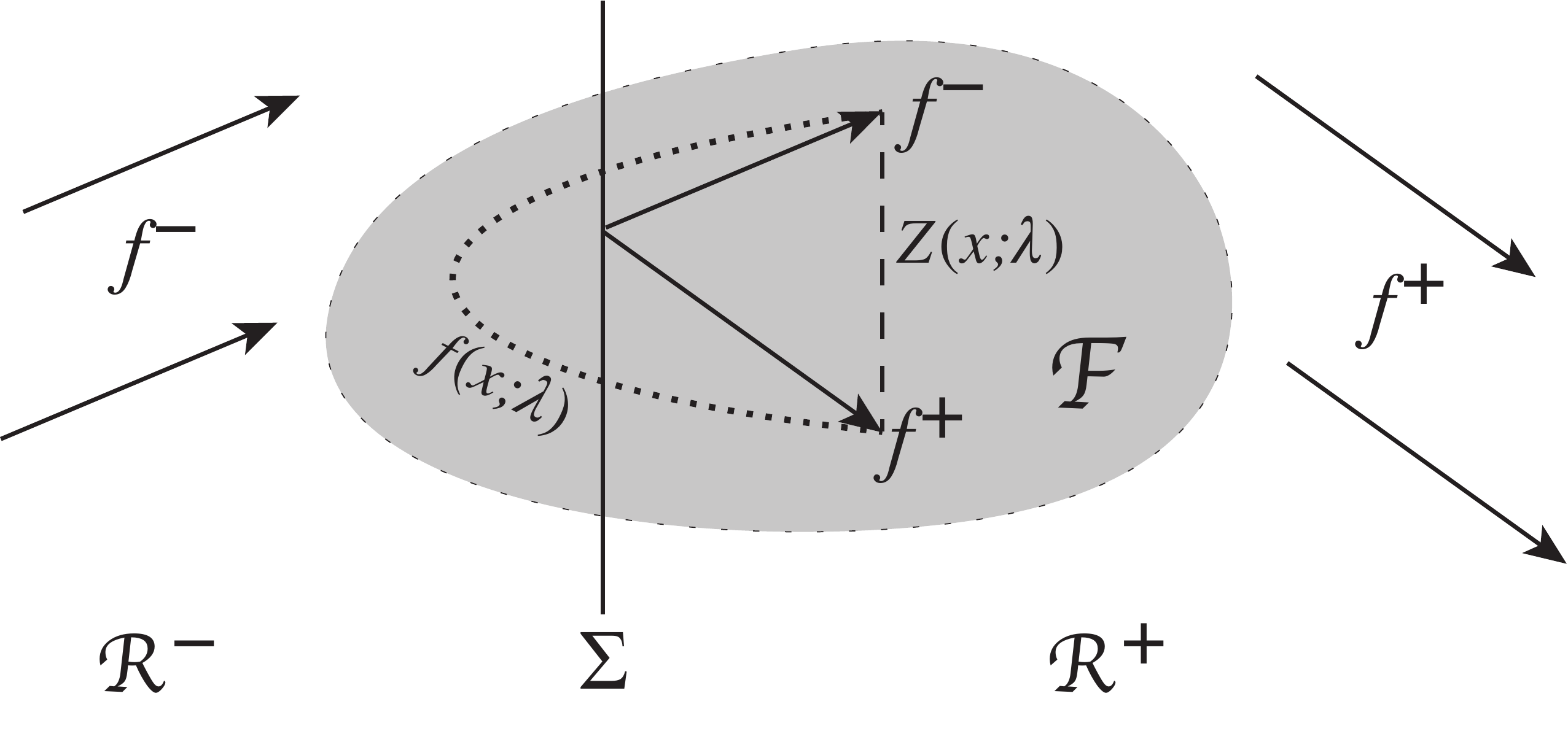}\vspace{-0.3cm}
\caption{\sf\footnotesize The vector field $f$ switches between $f^+$ and $f^-$ in regions $\op R^+$ and $\op R^-$. At the boundary $\Sigma$ Filippov considered either a general convex set $\op F$ containing $f^\pm$ (shaded area), or a convex hull $Z(x;\lambda)$ of $f^\pm$ (dashed line). The nonlinear combination $f(x;\lambda)$ allows us to explore $\op F$ more explicitly (dotted curve), by choosing a different $G$ we obtain a different curve of values $f(x;\lambda)\subset\op F$.}\label{fig:fil}
\end{figure}
\end{example}

Although Filippov (followed by many authors since) favoured (\ref{s2c}), it is worthwhile exploring the more general form (\ref{s2}), not least because in \cite{J,JSimp} it was shown to provide new ways of modeling real mechanical phenomena (namely static friction, the phenomenon that the force of dry-friction during sticking can exceed that during motion, not captured by applying Filippov's method to the basic discontinuous Coulomb friction law), and in \cite{H,Jiso} it is shown that similar nonlinearities become inescapable when multiple switches are involved (specifically it is shown that multiple switches create the possibility of multiple sliding solutions, which must be resolved by some kind of regularization or blow up of the discontinuity). It is therefore important obtain greater insight into the discontinuous dynamical systems represented by (\ref{s2}), one of the first concerns being typically their persistence within larger classes of systems. To this end it has been shown that the dynamics of (\ref{s2c}) persists when the discontinuity is {\it regularized} (i.e. smoothed) \cite{LST2} and, as we will show here, the same is equally true of the nonlinear combination (\ref{s2}).


The behaviours associated with adding $G$ in (\ref{s2}) have been referred to as {\it hidden dynamics}, because the first condition in (\ref{s2a}) means that $G$ vanishes for $h\neq0$, i.e. everywhere except at the discontinuity itself. The function $G$ may, for example, be any finite vector field multiplied by a scalar term like $\lambda(\lambda^2-1)$, $\sin(\lambda^2-1)$, or $\lambda^{2r}-1$ for any natural number $r$. 

In this paper we will consider how the nonlinear combinations (\ref{s2}) relate to singular limits of continuous systems via both regularization \cite{LST2}, and a converse to regularization known as pinching \cite{simic,DM}. 
We introduce both of these concepts below. 
Much of our analysis will concern the closeness of dynamics on $\Sigma$ in the discontinuous system (\ref{s2}) to invariant dynamics near $\Sigma$ in a topologically equivalent smooth system.

We set up the problem in Section \ref{sec:prelim}, then prove results regarding regularization and pinching in Sections \ref{sec:reg}-\ref{sec:pinch}. Brief remarks on blow-up, an alternative to both regularization and pinching which defines a dummy variable inside the discontinuity surface, are made in Section \ref{sec:blow}, with closing remarks in Section \ref{sec:conc}.

\section{Preliminaries: crossing or sliding in the nonlinear system}\label{sec:prelim}

The first step in studying (\ref{s2}) is to define more precisely what happens on $\Sigma$, our main interest being what happens when $G(x;\lambda)$ is allowed not to vanish there. We denote the interval of values taken by $\la$ as $\strip:=[-1,+1]$. 

Henceforth the symbol $p$ will always denote a point inside $\Sigma$, and where specific coordinates are useful we will sometimes let $h(x)=x_1$ and write $p=(0,\y)$. 

For any $p\in\Sigma$ we denote the component of $f$ normal to $\Sigma$ by the scalar function 
\begin{equation}\label{K}
\fn(p;\la):=f(p;\la)\cdot\nabla h(p)\;.
\end{equation}
This vanishes on the set 
\begin{equation}\label{slideset}
S(p):=\{\la^*\in\strip:\,\fn(p;\la^*)=0\},
\end{equation}
which may or may not have solutions for $\la^*\in\strip$. Places where there exist solutions to (\ref{slideset}) define regions where the vector field $f$ lies tangent to $\Sigma$ for one or more values of $\la^*\in\strip$, allowing the flow of (\ref{s2}) to {\it slide} along $\Sigma$, and we call  the set of all such points $p\in\Sigma$ the nonlinear {\it sliding region} $\Sigma^{ns}$, given by
\[
\Sigma^{ns}:=\{p\in\Sigma:\, S(p)\neq \emptyset\}. 
\]
The complement to this on $\Sigma$ is the set where (\ref{slideset}) has no solutions, so $f$ is transverse to $\Sigma$ for all $\lambda\in\strip$, defining the nonlinear {\it crossing region} $\Sigma^{nc}$, 
\[
\Sigma^{nc}:=\{p\in\Sigma:\, S(p)=\emptyset\}\;,
\]
such that $\Sigma=\ov{\Sigma^{ns}}\cup \ov{\Sigma^{nc}}$, ($\ov{\Sigma^{ns}}$ and $\ov{\Sigma^{nc}}$ denoting the closures of ${\Sigma^{ns}}$ and ${\Sigma^{nc}}$). 

The implication is that for $p\in\Sigma^{nc}$ the vector field $f(p;\la)$ pushes the flow transversally across $\Sigma$ between $\op R^+$ and $\op R^-$, while for $p\in\Sigma^{ns}$ the flow is able to slide along $\Sigma$. Substituting the solution $\lambda^*$ of (\ref{slideset}) into (\ref{s2}), the system that defines these nonlinear {\it sliding modes} is given by 
\begin{equation}\label{s4}
\dot p=f^{ns}(p):=f(p;\la^*(p))\;,\qquad\la^*(p)\in S(p)\;,
\end{equation}
with $f^{ns}$ defining the nonlinear {\it sliding vector field}. 
Typically there may exist a set of such functions $\lambda^*_i$, $i=1,2,...$, defining branches of solutions of $\fn(p;\lambda^*)=0$ in (\ref{slideset}), each on a subset $\sigma_i\subset\Sigma^{ns}$, such that the union of all $\sigma_i$'s covers $\Sigma^{ns}$, and $\lambda^*_i:\sigma_i\subset\Sigma^{ns}\mapsto\strip$. We then have a set of sliding modes specified by a set of equations defined by (\ref{s4}) on different branches $p\in\sigma_i$. 

If we fix $G\equiv0$ everywhere then the sliding region $\Sigma^{ns}$ and crossing region $\Sigma^{nc}$ are exactly the sliding and crossing regions defined by the Filippov's convention for the system (\ref{s2c}), which we therefore call the {\it linear crossing region} $\Sigma^c$ and {\it linear sliding region} $\Sigma^s$ (obtained directly by solving the above conditions neglecting $G$). The linear system (i.e. without $G$) can only have one (linear) sliding mode, on $\Sigma^{s}$, while the full system ($G$ nonzero on $\Sigma$) may have multiple (nonlinear) sliding modes as defined by (\ref{s4}) with (\ref{s2}). It is easily shown (see \cite{J}) that $\Sigma^s\subseteq\Sigma^{ns}$ and $\Sigma^{nc}\subseteq\Sigma^c$.

\section{Regularization}\label{sec:reg}

Let us first show that regularizations of the linear combination (\ref{s2c}) or of the nonlinear combination (\ref{s2}) can be related by a simple substitution. 

Let $C^r$ denote the class of $r$-times differentiable functions. 
We shall denote by 
\begin{equation*}\begin{array}{ll}
\psi:\R\rightarrow\R&\mbox{ a continuous function which is $C^1$ for $s\in(-1,1)$}\\
&\mbox{ such that $\psi(s)=\sgn(s)$ for $|s|\geq1$.}\\
&\mbox{We call $\psi$ a {\it transition function}}.\\
\phi:\R\rightarrow\R&\mbox{ a continuous function which is $C^1$ for $s\in(-1,1)$}\\
&\mbox{ such that $\phi(s)=\sgn(s)$ for $|s|\geq1$, and $\phi'(s)>0$ for $s\in(-1,1)$.}\\
&\mbox{We call $\phi$ a {\it monotonic transition function}.}
\end{array}\end{equation*}
 We also let
 \begin{equation}\label{phid}
\phi_{\de}(h):=\phi(h/\de)\qquad{\rm and}\qquad\psi_{\de}(h):=\psi(h/\de)\;.
\end{equation}

A regularization of a discontinuous system \eqref{s2c} or \eqref{s2} is a one--parameter family $Z_{\de}\in C^r$ for $r\geq0$ such that $f_{\de}$ converges to the discontinuous system when $\de\to 0$. The intention is that this represents a class of continuous functions approximated by (\ref{s1}) as $\de\to0$, the importance of (\ref{s2}) is that it will show this class to be larger than those derived from (\ref{s2c}). 
The Sotomayor-Teixeira method of regularization, see e.g. \cite{ST}, replaces $\lambda$ in \eqref{s2c} by a monotonic transition function $\phi$, to consider
\begin{equation*}
\dot x=\dfrac{1+\phi_{\de}(h(x))}{2}f^+(x)+\dfrac{1-\phi_{\de}(h(x))}{2}f^-(x).
\end{equation*}
We refer to this as a {\it linear}-regularization (or $\phi$--regularization in other references). 
It is shown in \cite{feck,BST,LST2} that this defines a system with slow invariant dynamics topologically equivalent to Filippov's (linear) sliding dynamics. One may ask what happens if we consider instead (\ref{s2c}) with a non-monotonic transition function $\psi$. When modeling a physical system, for example, there is no clear reason to exclude such possibilities, and we shall see below how they fit with established theory for discontinuous differential equations. 

We will show that the (non-monotonic) $\psi$ regularization of Filippov's linear combination (\ref{s2c}), 
\begin{equation}\label{nonreg}
\dot x=Z_{\de}(x):=\dfrac{1+\psi_{\de}(h(x))}{2}f^+(x)+\dfrac{1-\psi_{\de}(h(x))}{2}f^-(x)\qquad\qquad\qquad\qquad
\end{equation}
is equivalent to the (monotonic) $\phi$ regularization of a nonlinear combination (\ref{s2}), given by
$f_\delta(x)=f(x;\phi(h(x)/\de))$, i.e. 
\begin{equation}\label{reg}
\dot x=f_{\de}(x):=\dfrac{1+\phi_{\de}(h(x))}{2}f^+(x)+\dfrac{1-\phi_{\de}(h(x))}{2}f^-(x)+G(x;\phi_\de(h(x)))\;.
\end{equation}

\begin{theorem}\label{t1}
If $\phi$ is a monotonic transition function and $\psi$ is a non--monotonic transition function, then there exists a unique function $G(x;\la)$ satisfying (\ref{s2a}) such that the $\psi$--regularization of \eqref{s2c} is a $\phi$--regularization of \eqref{s2}. 
\end{theorem}
\begin{proof}
Let $\lambda=\phi(s)$, the function $\phi$ is monotonic in the interval $\strip$ and therefore has an inverse $s=\phi^{-1}(\lambda)$, so we can express $\psi$ in terms of $\la$ via a function $\Psi(\lambda)=\psi\left(\phi^{-1}(\lambda)\right)$. 
The $\psi$--regularization of (\ref{s2c}) as given by (\ref{nonreg}) can thus be re-arranged to
\[
\dot x=\dfrac{1+\lambda}{2}f^+(x)+\dfrac{1-\lambda}{2}f^-(x)+\left(\Psi(\lambda)-\lambda\right)\dfrac{f^+(x)-f^-(x)}{2}\;.
\]
If we define $G(x;\la)=\left(\Psi(\la)-\la\right)\left({f^+(x)-f^-(x)}\right)/{2}$, we obtain the nonlinear combination (\ref{s2}), and taking $\lambda=\phi_\delta(h(x))$ we obtain its $\phi$--regularization on $\lambda\in\strip$. Since for $|s|\geq1$ we have $\lambda=\phi(s)=\psi(s)=\sgn(s)$, this implies $G(x;\pm 1)=0$ as required by (\ref{s2a}). 
\end{proof}

A simple consequence of this is that the family of $\phi$--regularized nonlinear combinations (\ref{s2}) is larger than the family of $\psi$--regularized linear combinations (\ref{s2c}), as shown by the following.

\begin{corollary}\label{c1}
If $\phi$ is a monotonic transition function, then there exists a non--monotonic transition function $\psi$ such that the $\phi$--regularization of \eqref{s2} is a $\psi$--regularization of \eqref{s2c}, if and only if $G(x;\la)=\gamma(\la)\left({f^+(x)-f^-(x)}\right)/{2}$ such that $h(x)\gamma(\lambda)=0$. 
\end{corollary}
\begin{proof}
The proof follows directly by substituting $G$ into (\ref{reg}) and applying Theorem \ref{t1}.
\end{proof}

Figure \ref{fig:scheme} provides the resulting schematic of how the discontinuous systems and their regularizations considered above fit together. 
\begin{figure}[h!]\centering\includegraphics[width=0.7\textwidth]{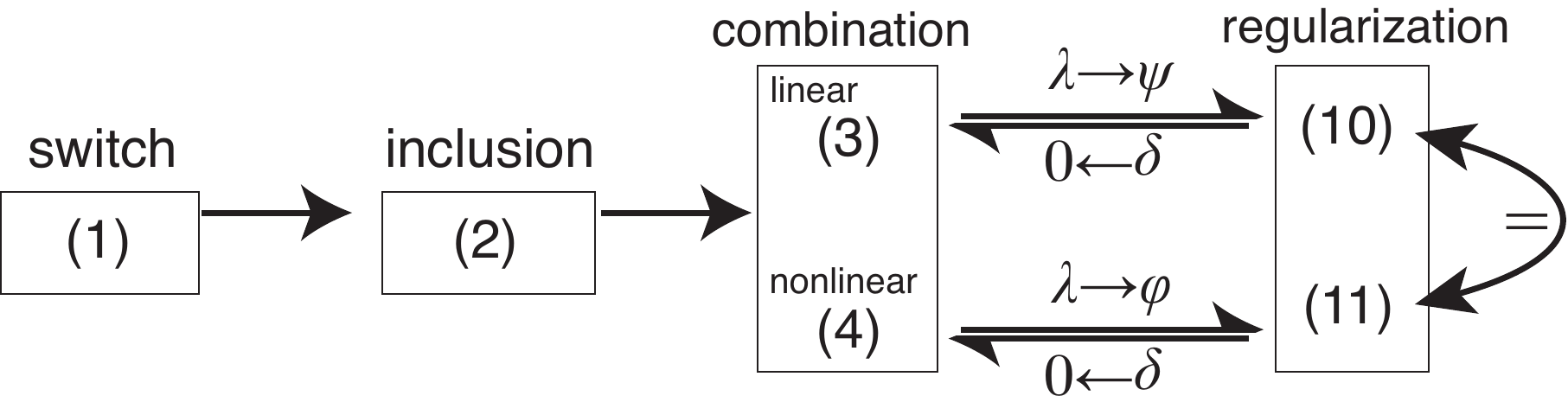}\vspace{-0.3cm}
\caption{\sf\footnotesize The discontinuous differential equation (\ref{s1}) is not defined on $\Sigma$, so is replaced by the inclusion (\ref{inc}), representing all possible systems at $\Sigma$. A solvable form for these is provided by the Filippov systems in the linear form (\ref{s2c}) or more general nonlinear form (\ref{s2}). In the following sections we applying a regularization of nonlinear or linear kind, yielding the differentiable systems (\ref{nonreg}) and (\ref{reg}) respectively, which are equivalent for some choice of transition functions $\phi_\delta$ and $\psi_\delta$, and conversely whose singular limits as $\delta\rightarrow0$ are (\ref{s2c}) and (\ref{s2}).}\label{fig:scheme}
\end{figure}

\bigskip

In the next theorem we extend the Sotomayor-Teixeira result to these systems, showing that the nonlinear regularization (\ref{reg}) exhibits slow invariant dynamics that is conjugate to the sliding modes of the discontinuous system (\ref{s2c}). The remainder of this section will consist of the proof of this theorem. First let us see how slow-fast dynamics arises in an example. 

\begin{example}
Consider the system
$$(\dot x_1,\dot x_2)=\frac{1+\lambda}2(1,-2)+\frac{1-\lambda}2(1,1)+(\lambda^2-1)(2,0)\;,$$
which is discontinuous if $\lambda={\rm sign}(x_1)$ for $x_1\neq0$. The regularization is obtained by replacing $\lambda\mapsto\phi_\delta(x_1)$ for small $\delta>0$. Figure \ref{fig:reg} shows the discontinuous system (left) with a nonlinear sliding region on which two sliding modes exist (one traveling upwards, the other downwards), and conjugate to each sliding mode. Compare this to the discontinuous linear and nonlinear systems in Example \ref{eg:fil}.  
\begin{figure}[h!]\centering\includegraphics[width=0.6\textwidth]{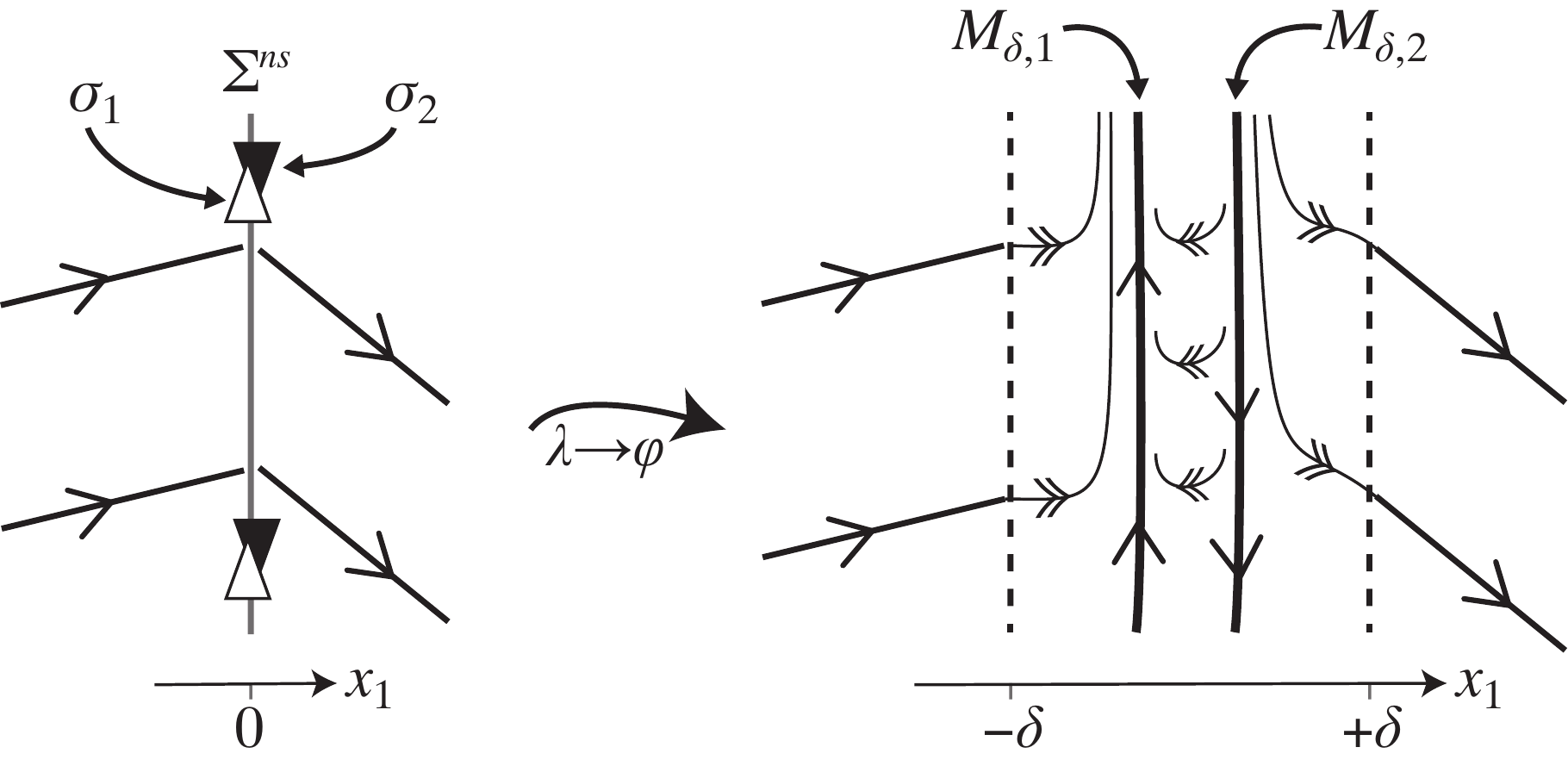}\vspace{-0.3cm}
\caption{\sf\footnotesize Left: a discontinuous system (\ref{s2}) with nonlinear sliding region with branches $\sigma_r$ for $r=1,2$ (white and black filled arrows). Right: the regularization in which each sliding branch $\sigma_k$ is conjugate to an invariant manifold $M_{\delta,k}$ of a slow-fast system (\ref{reg}). }\label{fig:reg}
\end{figure}
\end{example}

\begin{theorem}\label{t2}
Let the region $\sigma\subset\Sigma^{ns}$ be expressible as a graph $x_1=0$ in coordinates $x=(x_1,x_2,..,x_2)$, 
on which there exists a function $\la^*(p)$ such that $\fn(p;\la^*(p))=0$ in (\ref{K}) for every $p\in\sigma$. 
Then for any $C^r$ (or continuous) function $\phi$, the $\phi$--regularization contains a slow manifold $C^r$--diffeomorphic (homeomorphic) to $\sigma$, on which the slow dynamics is $C^r$--conjugated (topologically conjugated) to the nonlinear sliding dynamics (\ref{s4}). Moreover, if $\p \fn(p;\la^*(p))/\p \la\neq0$ then for $\de>0$ sufficiently small the nonlinear sliding dynamics defined on $\Sigma^{ns}$ persists to order $\de$, on a manifold $M_{\de}$ which is $\de$--close to $\Sigma^{ns}$.
\end{theorem}

\begin{proof}
In the coordinates given, $\sigma\subset\Sigma^{ns}$ is an open subset of the hyperplane $\{p=(0,x_2,x_3,\ldots,x_n)
\in\op D\}$. Writing vector components as $f=(f_1,f_2,...,f_n)$ for any function $f$, the normal component (\ref{K}) of the nonlinear combination (\ref{s2}) is
\begin{equation}\label{t1k}
\fn(p;\la)=\dfrac{1+\la}{2}f_1^+(p)+\dfrac{1-\la}{2}f_1^-(p)+G_1(p;\la).
\end{equation}
Sliding modes by (\ref{slideset})-(\ref{s4}) satisfy the differential-algebraic system
\begin{equation}\label{t1f}
\begin{array}{rl}
0=&\!\!\!f_1(p;\la^*(p))\\
\dot p_i=&\!\!\!
\dfrac{1+\la^*(p)}{2}f_i^+(p)+\dfrac{1-\la^*(p)}{2}f_i^-(p)+G_i(p;\la^*(p))
\end{array}
\end{equation}
for $i=2,3,...,n$. 

Now consider the $\phi$--regularization of \eqref{s2}, given by
\begin{equation}\label{t1r1}
\dot x_i=\dfrac{1+\phi_{\de}(x_1)}{2}f_i^+(x)+\dfrac{1-\phi_{\de}(x_1)}{2}f_i^+(x)+ G_i(x;\phi_{\de}(x_1))
\end{equation}
for $i=1,...,n$. 
By a change of variables to $u=x_1/\de$ and $v=(x_2,x_3,\ldots,x_n)$ for small $\delta>0$, we obtain
\begin{equation}\label{t1sp}
\begin{array}{rl}
\de \dot u=&\dfrac{1+\phi(u)}{2}f_1^+(u\de,v)+\dfrac{1-\phi(u)}{2}f_1^+(u\de,v)+ G_1(u\de,v;\phi(u)),\\
\dot x_i=&\dfrac{1+\phi(u)}{2}f_i^+(u\de,v)+\dfrac{1-\phi(u)}{2}f_i^+(u\de,v)+ G_i(u\de,v;\phi(u)),
\end{array}
\end{equation}
where $\delta$ is a singular perturbation parameter. 
In the limit $\delta=0$ we obtain the so-called reduced problem (using the notation $x=p$ on $\Sigma$)
\begin{equation}\label{t1rp}
\begin{array}{rl}
0=&\dfrac{1+\phi(u)}{2}f_1^+(p)+\dfrac{1-\phi(u)}{2}f_1^+(p)+ G_1(p;\phi(u))=\fn(p;\phi(u)),\\
\dot p_i=&\dfrac{1+\phi(u)}{2}f_i^+(p)+\dfrac{1-\phi(u)}{2}f_i^+(p)+ G_i(p;\phi(u)),\;\;\;i=2,...,n,
\end{array}
\end{equation}
which describes dynamics on the `slow' timescale $t$ (for standard concepts of singularly perturbed or slow-fast systems see \cite{fenichel,jones}). This dynamics inhabits a hypersurface called the {\it slow critical manifold}, defined implicitly by $0=\fn(p;\phi(u))$ in the first row of (\ref{t1rp}). 

%
By hypothesis there exists at least one function $\lambda^*(p)$ satisfying (\ref{slideset}), and therefore there exists at least one slow critical manifold $M_0$ given by the restriction $\phi(u)=\la^*(p)$. Since $\phi$ is invertible in $\strip$ and $\la^*(p)\in\strip$ for every $p\in\sigma$ we conclude that $M_0$ is the graph $u(p)=\phi^{-1}\circ\la^*(p)$. This is homeomorphic to $\sigma$ as we can let $H:\sigma\rightarrow M_0$ be the bijective function $H(0,v)=(\phi^{-1}\circ\la^*(0,v),v)$, for which $H(\sigma)=M_0$.
The function $H$ is invertible and its order of differentiability is the same as that of $\phi$.

Substituting $\phi(u)=\la^*(p)$ into \eqref{t1rp}, the reduced problem on $x_1=0$ becomes
\begin{equation}\label{t1rpr}
\begin{array}{rl}
\dot p_i=&\dfrac{1+\la^*(p)}{2}f_i^+(0)+\dfrac{1-\la^*(p)}{2}f_i^+(0)+ G_i(0;\la^*(p))=f_i(p;\la^*(p)),
\end{array}
\end{equation}
for $i=2,3,...,n$. 
Now let $\ov p=(0,\ov v)$, so if $t\mapsto x_t(\ov p)=(0,v(t,\ov p))$ is the solution of the nonlinear sliding mode \eqref{t1f} such that $x_0(\ov p)=\ov p\in\sigma$, then the solution $t\mapsto X_t(H(\ov p))$ of the reduced problem \eqref{t1rp} on the slow manifold such that $X_0(H(\ov p))=H(\ov p)$ is given by 
\[
X_t(H(\ov p))=(\phi^{-1}\circ\la^*(v(t,\ov v)),v(t,\ov v))=H(x_t(\ov p)).
\]
The flows of the regularized reduced (slow manifold) system and the discontinuous sliding system are therefore $C^r$(topologically)--conjugated.

It remains to show the persistence of the slow-fast dynamics for $\delta>0$. By rescaling time in (\ref{t1sp}) by $t=\de \tau$ and taking $\de\rightarrow0$, we obtain the so-called layer problem
\begin{equation}\label{t1lp}
\begin{array}{rl}
u'=&\dfrac{1+\phi(u)}{2}f_1^+(p)+\dfrac{1-\phi(u)}{2}f_1^+(p)+ G_1(p;\phi(u))=\fn(p;\phi(u)),\\
p_i'=&0,\qquad i=2,3,...,n.
\end{array}
\end{equation}
which prescribes dynamics on the fast timescale $\tau$ external to the slow manifolds. The slow manifold $M_0$ is a manifold of critical points of the layer problem, which is normally hyperbolic if $(\p \fn/\p \la)(p;\la^*(p))\neq0$. The existence of slow manifolds $\delta$--close to the slow critical manifold, with dynamics $\delta$--close to the reduced problem (\ref{t1sp}), then follows by Fenichel's theorem \cite{fenichel}.
\end{proof}

\section{Pinching}\label{sec:pinch}

Pinching, introduced in \cite{simic} and developed further in \cite{DM}, can be thought of as an inverse to regularization, providing a method of deriving a discontinuous system as an approximation to a continuous system. 
A region of state space is chosen, say some $|h|\le\e$ for $\e>0$, to be collapsed down to a manifold $\Sigma$ by means of a discontinuous transformation, resulting in a system of the form (\ref{s1}). 

In considering nonlinear switching systems we are able to put the notion of pinching on a more rigorous footing. To do so we must distinguish between {\it intrinsic pinching}, where the pinching parameter $\e$ is a small parameter of the original continuous system, and {\it extrinsic pinching} where the original problem is $\e$--independent. Before venturing into the technicalities, let us illustrate them with an example. 

\begin{example}\label{eg:hills}
Take a system 
\begin{equation}\label{hills}
(\dot x_1,\dot x_2)=\bb{-x_1,2\op H(x_1/\alpha;b)-1}\;,\qquad \op H(u;b)=\frac{u^b}{1+u^b}\;.
\end{equation}
The Hill function $\op H$ is a sigmoid graph with a switch about $h=x_1=0$, and is a function prevalent in biological applications (starting with \cite{Hill}). There is an invariant manifold along $x_1=0$ with dynamics $(\dot x_1,\dot x_2)=(0,-1)$. 

Let $b\gg1$ be fixed. We shall take discontinuous approximations of this system. First, assuming $\alpha$ and $b$ are constants, let us make an {\it extrinsic} pinching with respect to a small parameter $\e$ by transforming to a coordinate $\tilde x_1=h-\e{\rm sign}(h)$, creating a discontinuous system
\begin{equation}\label{hilltil}
(\dot{\tilde x}_1,\dot x_2)
=\bb{-\tilde x_1\mp\e,\;2\op H\bb{\frac{\tilde x_1\pm\e}\alpha;b}-1}
=\bb{-\tilde x_1\mp\e,\;2c_\pm-1+\ord{\tilde x_1}}
\end{equation}
where $c_\pm=\op H\bb{\pm\frac\e \alpha;b}$, with (\ref{hilltil}) taking the upper signs for $\tilde x_1>0$ and lower signs for $\tilde x_1<0$. If we fix $\alpha$ and pinch with respect to a small parameter $\e$ that is extrinsic to the smooth system (\ref{hills}), then expanding for small $\e/\alpha$ gives $c_\pm=\ord{\e/\alpha}$ and we can neglect it for small enough $\e$, giving the system in Figure \ref{fig:pinche}. Solving (\ref{slideset}) and (\ref{s4}) we obtain $\lambda^*=0$ and a sliding vector field $\dot p=(0,-1)+G_\e$ on $\tilde x_1=0$, which is equivalent to the dynamics on the invariant manifold $x_1=0$ of (\ref{hills}) with $G_\e\equiv0$. 
\begin{figure}[h!]\centering\includegraphics[width=0.5\textwidth]{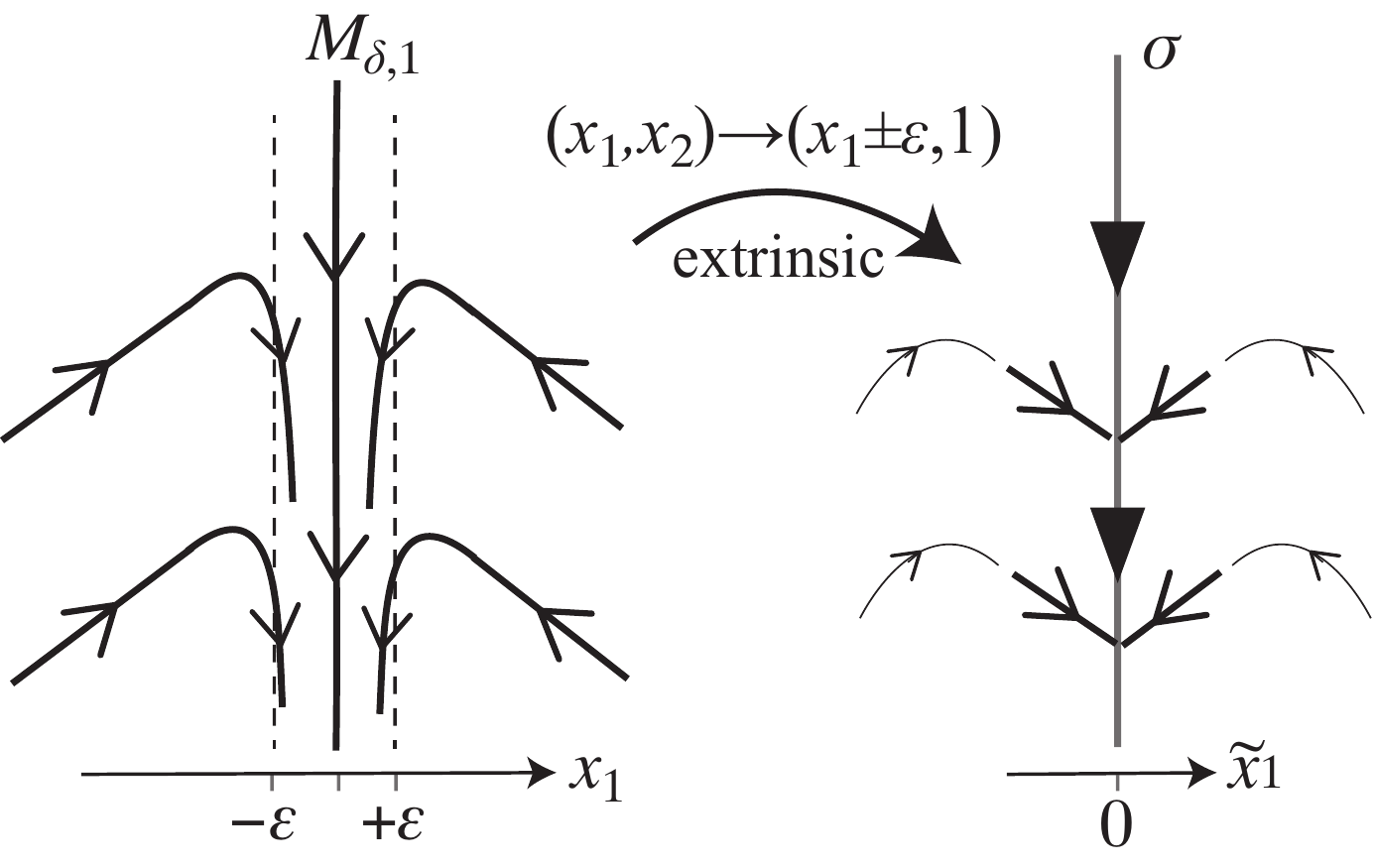}\vspace{-0.3cm}
\caption{\sf\footnotesize Differentiable systems with an invariant manifold $x_1=0$ (left), which we pinch by removing the region $|x_1|\le\e$, with $\e$ a small parameter extrinsic to (i.e. not appearing in) the smooth system. }\label{fig:pinche}
\end{figure}

Although the sliding mode captures the correction dynamics at $\tilde x_1=0$, the approximation outside is valid only for very small $\tilde x_1$ because is does not capture the turning around of the flow (the thin curves in the right of Figure \ref{fig:pinche}). To capture these we must use the exact expression in (\ref{hilltil}), so this approximation is quite weak. 

We can do something more powerful by pinching with respect to a parameter that is intrinsic to the system (\ref{hills}). If we set $\e=\alpha\sqrt2$ as an intrinsic pinching parameter, then expanding $\op H\bb{\pm\frac\e \alpha;b}$ for small $\alpha/\e$ gives $c_\pm=1+\ord{(\alpha/\e)^b}$, and we have the simple piecewise linear approximation $(-\tilde x_1\mp\e,1)$ for the righthand side of (\ref{hilltil}), as shown in the bottom row of Figure \ref{fig:pinchi}. The arrangement of the vector fields in the bottom right figure would give a linear sliding mode $\dot p=(0,1)$, which would be an incorrect representation of the dynamics of (\ref{hills}). Instead we need to find the nonlinear sliding mode, solving (\ref{slideset}) and (\ref{s4}) we obtain $\lambda^*=0$ and a sliding vector field $\dot p=(0,1)+G_\e$ on $\tilde x_1=0$, which is equivalent to dynamics on the invariant manifold $x_1=0$ in (\ref{hills}) if we set $G_\e=(0,-2)$, correctly capturing the dynamics of the smooth system. 

\begin{figure}[h!]\centering\includegraphics[width=0.5\textwidth]{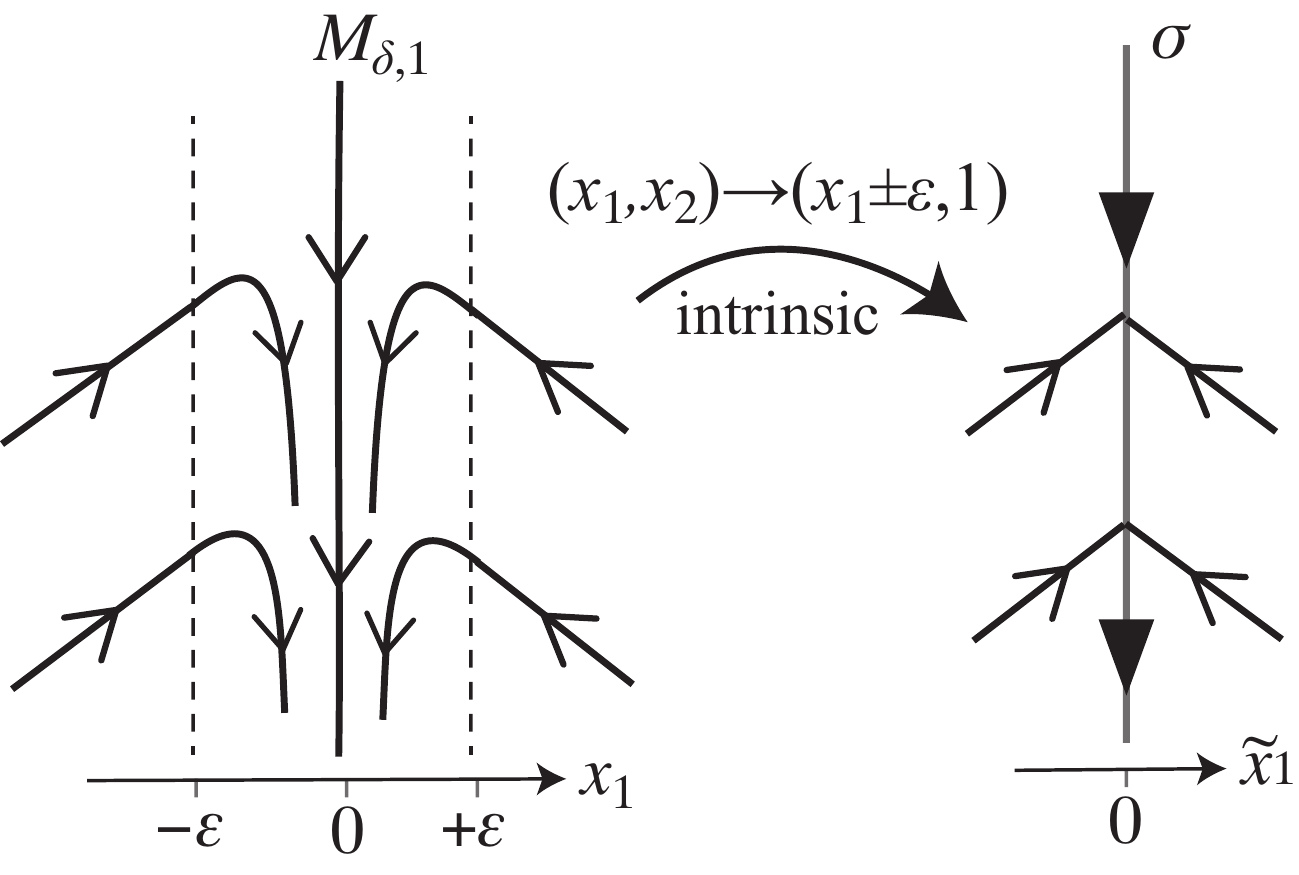}\vspace{-0.3cm}
\caption{\sf\footnotesize Starting from the same smooth system (left), we pinch by removing the region $|x_1|\le\e$, with $\e=\alpha\sqrt2$ and hence intrinsic to the smooth system. }\label{fig:pinchi}
\end{figure}

We say in these cases that $G_\e=(0,0)$ and $G_\e=(0,-2)$ {\it complete} the extrinsic and intrinsic systems, respectively. Below we generalize these ideas. 

%
%

\end{example}

\subsection{Extrinsic pinching}

Consider the dynamical system
\begin{equation}\label{cs}
\dot x=F(x),
\end{equation}
where $F$ is a $C^1$ function. Assume that the manifold $\Sigma=\{x\in\op D: h(x)=0\}$ is invariant under the flow, that is $F(p)\cdot\nabla h(p)=0$ for every $p\in \Sigma$. 

For small $\e>0$ consider the discontinuous system
\begin{equation}\label{eps}
\dot x=\left\{\begin{array}{l}
F(x+\e\nabla h(x))\quad \textrm{if}\quad h(x)>0,\\
F(x-\e\nabla h(x))\quad \textrm{if}\quad h(x)<0,
\end{array}\right.
\end{equation}
in which the manifold $\Sigma$ becomes a switching manifold between some $F^+(x;\e)=F(x+\e\nabla h(x))$ and some $F^-(x;\e)=F(x-\e\nabla h(x))$. We call (\ref{eps}) the {\it incomplete} extrinsically pinched system, ``incomplete'' because like (\ref{s1}) it is not yet well defined on $\Sigma$.

We then ask whether it is possible to complete the pinched system \eqref{eps} using a nonlinear combination (\ref{s2}), such that its nonlinear sliding modes (\ref{s4}) agree with the dynamics of \eqref{cs} on the invariant manifold $\Sigma$. When this is possible for some family of functions $G^{\e}$ ($G^\e$ being the nonlinear part for (\ref{s2}) now dependent on $\e$) we say that $G^{\e}$ {\it completes} the pinched system, and we call 
\begin{equation}\label{cps}
\begin{array}{l}
\dot x=f^{\e}(x;\la)=\dfrac{1+\la}{2}F(x+\e\nabla h(x))+\dfrac{1-\la}{2}F(x-\e\nabla h(x))+G^{\e}(x;\la),\\
\la\in\strip,\qquad h(x)G^{\e}(x;\la)=0\;,
\end{array}
\end{equation}
the {\it complete} extrinsically pinched system. 
In order to obtain $\lim_{\e\to 0}f^{\e}(x;\la)=F(x)$ we assume that the function $\e\mapsto G^{\e}(x;\la)$ is sufficiently differentiable and that $G^0(x;\la)=0$. 

Completing the pinched system in this way is possible provided that \eqref{cs} restricted to the manifold $\Sigma$ is structurally stable (see \cite{P}). The function $G$ that completes the pinched system is not unique.

\begin{theorem}\label{t3}
For $\e>0$ sufficiently small in (\ref{cps}), if there exists a continuous family $\la_{\e}^*(p)\in\strip$ of $C^1$ functions such that $\fn(p;\la_{\e}^*(p))=0$ by (\ref{slideset}) for every $p\in\Sigma$, then the nonlinear sliding mode by (\ref{s4}) satisfies
\[
\dot p=f^{ns}(p)=F(p)+r(p;\e)\qquad\textrm{on}\quad\Sigma^{ns},
\] 
where $r(p;\e)$ is a continuous function that is $C^1$ in the first variable, and where $r(p;\e)\to 0$ as $\e\to 0$. Moreover if we assume that \eqref{cs} restricted to the invariant manifold $\Sigma$ is structurally stable, then it is topologically equivalent to the nonlinear sliding dynamics. 
\end{theorem}

\begin{proof}
Direct application of (\ref{s4}) to (\ref{cps}) gives
\[
\begin{array}{rl}
f^{ns}(p)=&\!\!\!\dfrac{1+\la_{\e}^*(p)}{2}F(p+\e\nabla h(p))+\dfrac{1-\la_{\e}^*(p)}{2}F(p-\e\nabla h(p))+G^{\e}(p;\la_{\e}^*(p))\\
=&\!\!\!F(p)+r(p;\e),
\end{array}
\]
the second line following because $\la_{\e}^*(p)$ is a continuous family of functions.
Since the system $\dot p =F(p)$ is structurally stable it must therefore be topologically equivalent to $\dot p =f^{ns}(p)$.
\end{proof}

We shall assume now that the function $F$ is of class $C^{k+1}$, and that
\begin{equation}
G^{\e}(p;\la)=\sum_{i=1}^k\e^i \gamma_i(p;\la)+\CO(\e^{k+1})
\end{equation}
for some functions $\gamma_i$. 
Similar to (\ref{K}) we define the $\e$--family of functions $\fn^{\e}(p;\la)=f^{\e}(p;\la)\cdot\nabla h(p)$, and expand
\begin{equation}\label{Ksum}
\fn^{\e}(p;\la)=\sum_{i=r}^k \e^i \kappa_i(p;\la)/i!
\end{equation}
in terms of functions $\kappa_i$ given by
\begin{equation}
\begin{array}{rl}
\kappa_i(p;\la)=&i!\dfrac{\p^i}{\p\e^i}\fn^{\e}(p;\la)\Big|_{\e=0}\vspace{0.2cm}\\
=&\la^{\frac{1-(-1)^i}{2}}\left[(\nabla h(p)\cdot\nabla)^i F(p)\right]\cdot\nabla h(p)+\gamma_i(p;\la)\cdot\nabla h(p),
\end{array}
\end{equation}
for $i=1,2,\ldots,k$. Here $(\nabla h(p)\cdot\nabla)^i F(p)\in\op D$ denotes the scalar derivative $\displaystyle\nabla h\cdot\nabla=\sum_{j=1}^n\frac{\partial h}{\partial x_j}\frac{\partial\;}{\partial x_j}$ applied $i$ times to $F$ and evaluated at $p$. 

\begin{theorem}\label{t4}
For $r\leq k$ assume that $\kappa_i=0$ for $i=1,2,\ldots,r-1$ and $\kappa_r\neq0$. Suppose that there exists $\ell(p)\in(-1,1)$ such that $\kappa_r(p;\ell(p))=0$ and $(\p \kappa_r/\p \la)(p;\ell(p))\neq0$ for every $p\in\Sigma$. Then for $\e>0$ sufficiently small there exists a continuous family $\la_{\e}^*(p)\in\strip$ of $C^1$ functions such that $\fn^{\e}(p;\la_{\e}^*(p))=0$ for every $p\in\Sigma$. Moreover if we assume that the system \eqref{cs} restricted to the invariant manifold $\Sigma$ is structurally stable, then on $\Sigma$ it is topologically equivalent to the nonlinear sliding mode defined by (\ref{s4}).
\end{theorem}

\begin{proof}
Assuming that $\kappa_i=0$ for $i=1,2,\ldots,r-1$ we write using (\ref{Ksum})
\[
\fn^{\e}(p;\la)=\e^r\dfrac{\kappa_r(p;\la)}{r!}+\CO(\e^{r+1}).
\]
Since $\kappa_r(p;\ell(p))=0$ and $(\p \kappa_r/\p \la)(p;\ell(p))\neq0$, applying the implicit function theorem for the function $\fn^{\e}(p;\la)/\e^r$ we obtain, for $\e>0$ sufficiently small, the existence of a differentiable family $\la_{\e}^*(p)\in\strip$ of $C^1$ functions such that $\fn^{\e}(p;\la_{\e}^*(p))=0$ for every $p\in\Sigma$.
The result follows by applying Theorem \ref{t3}.
\end{proof}

In some cases it is sufficient to take $G^{\e}\equiv0$ (i.e. a linear combination) to complete the pinched system \eqref{eps}. The following corollary concerns cases, as in Example \ref{eg:hills}, for which $G^{\e}$ cannot be zero everywhere.

\begin{corollary}\label{c2}
Assume in (\ref{eps}) that $F$ is a $C^3$ function. The following statements hold:
\begin{itemize}
\item[$(a)$] If $[\nabla h(p)\cdot\nabla F(p)]\cdot\nabla h(p)\neq0$ then $G^{\e}\equiv0$ completes the pinched system (\ref{cps}).

\item[$(b)$] If $[(\nabla h(p)\cdot\nabla) F(p)]\cdot\nabla h(p)=0$ and $[(\nabla h(p)\cdot\nabla)^2F(p)]\cdot\nabla h(p)\neq0$ then the function $G^{\e}\equiv0$ does not complete the pinched system. In this case $G^{\e}=\e^2(\la^2-1)C(p)
$ with $C(p)\neq[(\nabla h(p)\cdot\nabla)^2 F(p)]\cdot\nabla h(p)$ completes the system.
\end{itemize}
\end{corollary}

\begin{proof}
Taking $G^{\e}\equiv0$ we have from above that
\[
\fn^{\e}(p;\la)=\e\la[(\nabla h(p)\cdot\nabla) F(p)]\cdot\nabla h(p)+\e^2\dfrac{1}{2}[(\nabla h(p)\cdot\nabla)^2 F(p)]\cdot\nabla h(p)+\CO(\e^3).
\]
If $[\nabla h(p)\cdot\nabla F(p)]\cdot\nabla h(p)\neq 0$ we can choose $\ell(p)=0$, thus $\kappa_1(p,0)=0$ and $(\p \kappa_1/\p \la)(p;\ell(p))=[\nabla h(p)\cdot\nabla F(p)]\cdot\nabla h(p)\neq0$. Hence applying Theorem \eqref{t4} we have statement $(a)$.

If instead $[\nabla h(p)\cdot\nabla F(p)]\cdot\nabla h(p)= 0$ and $[(\nabla h(p)\cdot\nabla)^2 F(p)]\cdot\nabla h(p)\neq0$, there is no bounded family of solutions $\la_{\e}^*(p)$ of the equation $\fn^{\e}(p;\la_{\e}^*(p))=0$ for $G^{\e}\equiv0$. 
Taking instead $G^{\e}=\e^2(\la^2-1)C(p)
$ such that $C(p)\neq[(\nabla h(p)\cdot\nabla)^2 F(p)]\cdot\nabla h(p)$ we have that
\[
\fn^{\e}(p;\la)=\e^2C(p).
\]
So $\la_{\e}^*(p)=0\in\strip$ is a family of solutions of $\fn^{\e}(p;\la_{\e}^*(p))=0$. Applying Theorem \ref{t4} we then have statement $(b)$.
\end{proof}

A simple example is given by $\dot x_1=-x_1$ with $(\dot x_2,..,\dot x_n)=q(x_2,...,x_n)$ where $q$ is any smooth function; this would give a complete pinched system with Filippov (i.e. $G^{\e}\equiv0$) sliding dynamics equivalent to the smooth system's invariant dynamics on $x_1=0$. Instead consider the following more interesting system. 

\begin{example}\label{eg:s2}
For $x_1\in\R$ and $\y=(x_2,x_3,...,x_n)\in\R^{n-1}$ consider the system
\begin{equation}
\dot x_1=x_1^2,\quad\dot{\y}=q(\y).
\end{equation}
Taking $h(x_1,\y)=x_1$ the manifold $\Sigma=\cc{x\in\op D:x_1=0}$ is invariant under the flow. The dynamics defined on $\Sigma$ is given by $\dot{\y}=q(\y)$, and the incomplete pinched system is given by
\begin{equation}\label{epseg}
\dot x_1=\left\{\begin{array}{l}(x_1+\e)^2\quad\textrm{if}\quad x_1>0\\(x_1-\e)^2\quad\textrm{if}\quad x_1<0\end{array}\right\},  \quad \dot{\y}=q(\y).
\end{equation}
Computing the function $\fn^{\e}(0,\y;\la)$ we obtain $\fn^{\e}(0,\y;\la)=G^{\e}(0,\y;\la)\cdot\nabla h+\e^2$. Clearly for $G^{\e}\equiv0$ (the linear/Filippov case) with $\e>0$ the equation $\fn^{\e}(0,\y;\la)=0$ has no solutions, instead \eqref{eps} has only crossing solutions, and this does not represent the dynamics of the smooth system (\ref{epseg}). Taking instead $G^{\e}(x_1,\y;\la)=(\e^2(\la^2-1),0,0,...)$ we find that, for $\e>0$ sufficiently small, $\la_{\e}^*(x_1,\y)=0\in\strip$ is a family of solutions of $\fn^{\e}(0,\y;\la)=0$, and produces a nonlinear sliding mode given from (\ref{s4}) by $\dot{\y}=f^{ns}(0,\y)=\bb{0,q(\y)}$. 
\end{example}

In this example, therefore, we can complete the pinched system, but we cannot use Theorem \ref{t4} to prove equivalence between the pinched sliding dynamics and the original invariant dynamics on $\Sigma$, because the original continuous system, in particular the term $\dot x_1=x_1^2$, is structurally unstable. To handle such cases it is necessary to perturb the original system by a small quantity. It is then natural to pinch with respect to that small quantity, giving a pinching parameter that is intrinsic to the system.

\subsection{Intrinsic pinching}

Let $I$ and $U$ be open bounded subsets of $\R$ and $\R^{n-1}$, respectively. For $x_1\in I$ and $\y=(x_2,x_3,...,x_n)\in U$ consider the system
\begin{equation}\label{ics}
\dot x_1=F_1(x_1,\y;\mu),\quad\dot {\y}=\mu\,F_\y(x_1,\y;\mu).
\end{equation}
where $F=(F_1,F_\y)$ is a $C^1$ function and $\mu$ is a small parameter. We assume that for $\mu=0$ the graph $\Sigma=\{(0,\y):\,\y\in U\}$ is a critical invariant manifold of \eqref{ics}, that is $F_1(0,\y;0)=0$ for every $\y\in U$.

We also assume that, for $\mu>0$ sufficiently small, the graphs $\Sigma_{\e}^i=\{(m_{\e}^i(\y),\y):\, \y\in U\}$ for $i=1,2,\ldots,k,$ are invariant manifolds of \eqref{ics}, where $m_{\e}^i(\y)=\e\, m_i(\y)+\CO(\e^2)$ for some differentiable functions $m_i:\ov U\rightarrow \R$, such that the $\Sigma^i_\e$ are order $\e$-perturbations of $\Sigma$. We assume that $\mu=\CO(\e^r)$ where $r\ge1$, so that taking $\mu=\mu(\e)$ we have that $\mu(0)=0$. System \eqref{ics} induces dynamics on each $\Sigma_{\e}^i$, namely
\begin{equation}\label{is}
\dot{\y}=\mu(\e) F_\y(m_{\e}^i(\y),\y;\mu(\e))
\qquad{\rm on}\qquad x_1=m_{\e}^i(\y)\;.
\end{equation}

Now let $\man$ be a positive real number such that $\man>\max\{|m_i(\y)|:\,\y\in \ov U,\,i=1,2,\ldots,k\}$. For $\e>0$ sufficiently small we consider the following discontinuous system,
\begin{equation}\label{iips}
\begin{array}{l}
\dot x_1=\left\{\begin{array}{l}
F_1(x_1+\e\man,\y;\mu(\e))\quad \textrm{if}\quad x_1>0,\\
F_1(x_1-\e\man,\y;\mu(\e))\quad \textrm{if}\quad x_1<0,
\end{array}\right.\vspace{0.2cm}\\
\dot{\y}=\left\{\begin{array}{l}
\mu(\e)F_\y(x_1+\e\man,\y;\mu(\e))\quad \textrm{if}\quad x_1>0,\\
\mu(\e)F_\y(x_1-\e\man,\y;\mu(\e))\quad \textrm{if}\quad x_1<0.
\end{array}\right.
\end{array}
\end{equation}
We call a {\it incomplete} intrinsically pinched system, where $\Sigma$ is now the switching manifold where the dynamics is not well defined. The discontinuous vector field $(F_1,\mu F_\y)$ on the righthand side of (\ref{iips}) will be denoted by $F(x_1,\y;\mu(\e))$.

As we did for extrinsic pinching, we must now attempt to complete the system. In this case we must ask whether the pinched system \eqref{iips} can be completed in the form \eqref{s2} such that there exist $k$ nonlinear sliding modes, each of which agrees with the dynamics of \eqref{is} for $i=1,2,\ldots,k$. When this is possible for some family of functions $G^{\e}$ we say that $G^{\e}$ {\it completes} the pinched system, and we call
\begin{equation}
\begin{array}{rl}
\dot x_1=&\!\!\!f^{\e}(x_1,\y;\la)\\
=&\!\!\!\dfrac{1+\la}{2}F(x_1+\e\man,\y;\mu(\e))+\dfrac{1-\la}{2}F(x_1-\e\man,\y;\mu(\e))+G^{\e}(x_1,\y;\la),\\
\la\in&\!\!\!\strip,\qquad G^{\e}(x_1,\y;\pm 1)=0,
\end{array}
\end{equation}
the {\it complete} intrinsically pinched system. 
As before we impose $G^{0}(x;\lambda)=0$.

\begin{theorem}\label{thm:m1}
Suppose that the system \eqref{ics} has an invariant manifold defined as the graph of the function $m_{\e}(\y)=\e m(\y)+\CO(\e^2)$. 
If the system 
\[
\begin{array}{rl}
\dot{\y}=&\!\!\!\e\mu'(0)F_\y(0,\y;0)+\dfrac{\e^2}{2}\Big(\mu''(0)F_\y(0,\y;0)+2\mu'(0)^2\dfrac{\p F_\y}{\p\mu}(0,\y;0)\\
&\;\;\hspace{5.5cm}+\;2\mu'(0)m_i(\y)\dfrac{\p F_\y}{\p x_1}(0,\y;0)\Big)
\end{array}
\]
is structurally stable and
\[
\mu'(0)\dfrac{\p F_1}{\p \mu}(0,\y;0)\dfrac{\p F_1}{\p x_1}F_1(0,\y;0)\neq0,
\]
then the function $G^{\e}(x_1,\y;\la)=(0,0)$ completes the system.
\end{theorem}

\begin{proof}
The graph $\Sigma_{\e}=\{(m_{\e}(\y),\y):\y\in U\}$ is an invariant manifold for system \eqref{ics}, so taking $h_{\e}(x_1,\y)=x_1-m_{\e}(\y)$ we have  
\[
\begin{array}{rl}
0=&\nabla h_{\e}(m_{\e}(\y),\y)F(m_{\e}(\y),\y;\mu(\e))\\
=&F_1(m_{\e}(\y),\y;\mu(\e))-\mu(\e) m_{\e}'(\y) F_\y(\e,m(\y),\y;\mu(\e)),
\end{array}
\] 
for $\e>0$ sufficiently small. Thus taking the derivative in $\e=0$ we obtain 
\begin{equation}\label{eq}
\mu'(0)\dfrac{\p F_1}{\p \mu}(0,\y;0)+m(\y) \dfrac{\p F_1}{\p x_1}(0,\y;0)=0.
\end{equation}
As previously we define 
\[
\begin{array}{rl}
\fn^{\e}(0,\y;\la)=&\nabla h(0,\y) f^{\e}(0,\y;\la)\\
=&\dfrac{1+\la}{2}F_1(\e\, \man,\y;\mu(\e))+\dfrac{1-\la}{2}F_1(-\e\, \man,\y;\mu(\e))\\
=&\e\left(\mu'(0)\dfrac{\p F_1}{\p \mu}(0,\y;0)+\man\,\la \dfrac{\p F_1}{\p x_1}(0,\y;0)\right)+\CO(\e^2).
\end{array}
\]
Now let $\CK(\y;\la,\e)=\dfrac{\fn^{\e}(0,\y;\la)}{\e}$. From \eqref{eq} we have that $\CK\left(\y;\dfrac{m(\y)}{\man},0\right)=0$, and by hypothesis
\[
\dfrac{\p \CK}{\p \la}(\y;\la,0)\Big|_{(\la,\e)=(m_i(\y)/\man,0)}=\man\dfrac{\p F_1}{\p x_1}(0,\y;0)\neq0.
\]
Hence from the implicit function theorem we have that for $\e>0$ sufficiently small there exists $\la(0,\y;\e)= \dfrac{m(\y)}{\man}+\e\ov{\la}+\CO(\e^2)$ such that $\la(0,\y;\e)\in\strip$ and $\fn^{\e}(0,\y;\la(0,\y;\e))=0$ for every $\y\in U$ and for $\e>0$ sufficiently small. It is easy to obtain an expression for $\ov{\la}$, but we do not require it here.

Writing $f=(f_1,f_\y)$, the nonlinear sliding mode $f(0,\y;\la(0,\y;\e))=(0,f_\y(0,\y;\la(0,\y;\e)))$ is given by
\begin{equation}\label{nlsm}
\begin{array}{rl}
f_\y(0,\y;\la_i(0,\y;\e))=&\e\mu'(0) F_\y(0,\y;0)+\e^2\Big(\dfrac{\mu''(0)}{2}F_\y(0,\y;0)\\
&+\mu'(0)^2\dfrac{\p F_\y}{\p \mu}(0,\y;0)+\mu'(0)m(\y)\dfrac{\p F_\y}{\p x_1}(0,\y;0)\Big)+\CO(\e^3).
\end{array}
\end{equation}
Hence, expanding system \eqref{is} about $\e=0$ in a Taylor series up to second order in $\e$, we conclude that the nonlinear sliding mode \eqref{nlsm} is equivalent to the system \eqref{is}.
\end{proof}

A prototype for systems satisfying the hypotheses of Theorem \ref{thm:m1} is $\dot x_1=x_1-\mu$, $\dot x_2=\mu x_2$, with a slow invariant manifold $x_1=\mu m_\mu(x_2)$ that becomes the critical manifold $x_1=0$ when $\mu=0$.

It is clear that the function $G^{\e}\equiv0$ does not complete the system if $k>1$. In particular we have the following.

\begin{theorem}\label{thm:m2}
Suppose that system \eqref{ics} has two invariant manifolds defined as the graphs of the functions $m_{\e}^i(\y)=\e m_i(\y)+\CO(\e^2)$ for $i=1,2$ where $\mu(\e)=\CO(\e^2)$. We assume $m_1\neq m_2$ and that
\[
\mu''(0)\dfrac{\p F_1}{\p \mu}(0,\y;0)\dfrac{\p^2 F_1}{\p x_1^2}F_1(0,\y;0)\neq0.
\]
If for $\e>0$ sufficiently small the system 
\[
\dot{\y}=\dfrac{\mu''(0)}{2}F_\y(0,\y;0)+\dfrac{\e}{6}\left(\mu'''(0)F_\y(0,\y;0)+3\mu''(0)m_i(\y)\dfrac{\p F_\y}{\p x_1}(0,\y;0)\right)
\]
is structurally stable for $i=1,2$, then the function 
\[
G^{\e}(x_1,\y;\la)=\e^2(\la^2-1)\left(\dfrac{\man^2}{2}\dfrac{\p^2F_1}{\p x_1^2}(0,\y;0)\,,\, 0\right)
\]
completes the system.
\end{theorem}
\begin{proof}
The graph $\Sigma_{\e}^i=\{(m_{\e}^i(\y),\y):\,\y\in U\}$ is an invariant manifold for system \eqref{ics}, so taking $h_{\e}^i(x_1,\y)=x_1-m_{\e}^i(\y)$ we have that 
\[
\begin{array}{rl}
0=& F(m_i(\y),\y;\mu(\e))\cdot \nabla h_i(m_i(\y),\y)\\
=&F_1(\e m_i(\y),\y;\mu(\e))-\e m_i'(\y) F_\y(\e,m(\y),\y;\mu),
\end{array}
\] 
for $\e>0$ sufficiently small. Thus taking the second derivative at $\e=0$ we obtain 
\begin{equation}\label{eqp}
\mu''(0)\dfrac{\p F_1}{\p \mu}(0,\y;0)+m_i(\y)^2 \dfrac{\p^2 F_1}{\p x_1^2}(0,\y;0)=0.
\end{equation}
As previously we define 
\[
\begin{array}{rl}
\fn^{\e}(0,\y;\la)=& f^{\e}(0,\y;\la)\cdot \nabla h(0,\y)\\
=&\dfrac{1+\la}{2}F_1(\e\, \man,\y;\mu(\e))+\dfrac{1-\la}{2}F_1(-\e\, \man,\y;\mu(\e))+G^{\e}(0,\y;\la)\\
=&\dfrac{\e^2}{2}\left(\mu''(0)\dfrac{\p F_1}{\p \mu}(0,\y;0)+\man^2\,\la^2 \dfrac{\p^2 F_1}{\p x_1^2}(0,\y;0)\right)+\CO(\e^2).
\end{array}
\]
Now let $\CK(\y;\la,\e)=\dfrac{\fn^{\e}(0,\y;\la)}{\e^2}$. From \eqref{eqp} we have $\CK\left(\y;\dfrac{m_i(\y)}{\man},0\right)=0$, and by hypothesis
\[
\dfrac{\p \CK}{\p \la}(\y;\la,\e)\Big|_{(\la,\e)=(m_i(\y)/\man,0)}=\man\,m_i(\y)\dfrac{\p^2 F_1}{\p x_1^2}(0,\y;0)\neq0.
\]
Hence from the implicit function theorem, for $\e>0$ sufficiently small there exists $\la_i(0,\y;\e)= \dfrac{m_i(\y)}{\man}+\CO(\e)$ such that $\la_i(0,\y;\e)\in\strip$ and $\fn^{\e}(0,\y;\la_i(0,\y;\e))=0$ for every $\y\in U$ and for $i=1,2$.

The nonlinear sliding mode $f(0,\y;\la_i(0,\y;\e))=(0,f_\y(0,\y;\la_i(0,\y;\e)))$ is given by
\begin{equation}\label{nlsmp}
\begin{array}{rl}
f_\y(0,\y;\la_i(0,\y;\e))=&\dfrac{\e^2\mu''(0)}{2}F_\y(0,\y;0)+\dfrac{\e^2}{6}\left(\mu'''(0)F_\y(0,\y;0)\right.\\
&\left.+3\mu''(0)m_i(\y)\dfrac{\p F_\y}{\p x_1}(0,\y;0)\right)+\CO(\e^4),
\end{array}
\end{equation}
for $i=1,2$. Hence, expanding system \eqref{is} around $\e=0$ in Taylor series up to third order in $\e$, we conclude that the nonlinear sliding mode \eqref{nlsmp} is equivalent to the system \eqref{is} for each $i=1,2$.
\end{proof}

A prototype for systems satisfying the hypotheses of Theorem \ref{thm:m2} is $\dot x_1=x_1^2-\mu$, $\dot x_2=\mu x_2$, with slow manifolds $x_1=\pm\sqrt\mu m(x_2)$ which are normally hyperbolic for $\mu>0$, but which coalesce onto a non-hyperbolic critical manifold $x_1=0$ for $\mu=0$.

\section{Blow-up}\label{sec:blow}

In regularization we replaced the switching parameter $\lambda$ with a differentiable function $\phi_\delta(h(x))$. An alternative to this is to consider $\lambda$ itself as a variable on the surface $\Sigma$, and the way to use this to resolve nonlinear sliding modes was discussed in \cite{J}. For completeness a few remarks are pertinent here. 

Taking regularization as a motivation, let us say that $\lambda$ can be expressed as the limit of a function $\displaystyle\lambda=\lim_{\delta\to0}\phi(h/\delta)$, then $$\dot\lambda=\frac1\delta\dot h\lim_{\delta\to0}\phi'(h/\delta)=\frac1{\tilde\delta}f\cdot\nabla h$$ where $\tilde\delta=\delta/\lim_{\delta\to0}\phi'(h/\delta)$
. Assuming that $\phi$ is monotonically increasing on $h\in(-1,+1)$, that $\tilde \delta$ is finite, and moreover that there exist $r>0$ and $R>0$ with $0<r<\delta\ll1$, such that $\tilde\delta<R$ for $h\in(-1+r,+1-r)$. (For example, for a piecewise-continuous transition function where $\phi(h/\delta)=h/\delta$ on $|h|\le\delta$ we have $r$ arbitrarily small and $R=1$). We denote the time derivative with respect to $t/\tilde\delta$ by a prime, so $\tilde\delta\frac{d\;}{dt}\lambda\equiv\lambda'$, then let $\delta\to0$ and thus obtain a two-timescale system on $\Sigma$ given by
\begin{equation}
\begin{array}{rl}
\lambda'=&\!\!\!f_1(p;\la)=\dfrac{1+\la}{2}f_1^+(p)+\dfrac{1-\la}{2}f_1^-(p)+G_1(p;\la),\\
\dot p_i=&\!\!\!f_i(p;\la)=\dfrac{1+\la}{2}f_i^+(p)+\dfrac{1-\la}{2}f_i^-(p)+G_i(p;\la),
\end{array}
\end{equation}
$i=2,3,...,n$. The existence of slow invariant manifolds and their correspondence to sliding dynamics can be established similarly to the procedure for regularization in Section \ref{sec:reg}.

\section{Closing remarks}\label{sec:conc}

The equivalence between singular limits of continuous systems $\dot x=f(x)$, and discontinuous systems (\ref{s1}) resolved either as the traditional linear combination (\ref{s2c}) or its nonlinear extension (\ref{s2}), 
helps us understand their robustness as models of physical systems. In essence this states that systems which are only piecewise continuous are structurally stable to perturbations that smooth out their discontinuities. Although intuitively acceptable, this notion is not trivial, and nonlinear sliding modes make far richer dynamics possible close to the discontinuity than are apparent if nonlinearities are neglected. 

Particular forms for the function $G^{\e}$ that complete an intrinsically pinched system are given here for slow-fast dynamics with one or two slow critical invariant manifolds, but the result can certainly be extended, and a general theory may proceed along similar lines to normal forms of singularities, see e.g. \cite{ps}.

\section*{Acknowledgements}

MRJ is supported by EPSRC grant EP/J001317/2.  DDN is supported by a FAPESP--BRAZIL grant 2012/10231-7.


\begin{thebibliography}{99}


\bibitem{feck} {\sc J. Awrejcewicz and M. Fe\v ckan and P. Olejnik} {\it On continuous approximation
of discontinuous systems} (2005) Nonlinear Analysis {\bf 62} 1317--1331

\bibitem{simic} {\sc M. E. Broucke and C. C. Pugh and S. N. Simi\'c} {\it Structural stability of piecewise smooth systems} (2001) Comput. Appl. Math. {\bf 20} 51--90

\bibitem{BST} \textsc{C. Buzzi, P.R. da Silva and M.A. Teixeira},
\textit{A singular approach to discontinuous vector fields on the plane} (2006) J. Diff. Equations {\bf
231} 633--655

\bibitem{DM} {\sc M. Desroches, M.R. Jeffrey} (2011) {\it Canards and curvature: nonsmooth approximation by pinching} Nonlinearity {\bf 24} 1655--1682

\bibitem{bc08} \textsc{M. di Bernardo, C. J. Budd, A. R. Champneys and P. Kowalczyk},
\textit{Piecewise-Smooth Dynamical Systems: Theory and Applications} (2008) Springer

\bibitem{fenichel} {\sc N. Fenichel} (1979) {\it Geometric singular perturbation theory for ordinary differential equations} J. Differential Equations {\bf 31}(1):53--98

\bibitem{F} \textsc{A. F. Filippov},
\textit{Differential equations with discontinuous righthand side}, Mathematics and Its Applications (1988) Kluwer Academic Publishers, Dordrecht

\bibitem{H} {\sc N. Guglielmi and E. Hairer} {\it Classification of hidden dynamics in discontinuous dynamical systems} (2015) SIADS, in press February 2015

\bibitem{Hill} {\sc A. V. Hill} {\it The possible effects of the aggregation of the molecules of haemoglobin on its dissociation curves} (1910) Proc. Physiol. Soc. {\bf 40} iv-vii

\bibitem{J} {\sc M. R. Jeffrey} {\it Hidden dynamics in models of discontinuity and switching} (2014) Physica D {\bf 273-274} 34--45

\bibitem{Jiso} {\sc M. R. Jeffrey} {\it Dynamics at a switching intersection: hierarchy, isonomy, and multiple-sliding} (2014) SIADS {\bf 13} (3) 1082-1105

\bibitem{JSimp} {\sc M. R. Jeffrey and D. J. W. Simpson} {\it Non-Filippov dynamics arising from the smoothing of nonsmooth systems, and its robustness to noise} (2013), Nonlinear Dynamics {\bf 76}(2) 1395--1410

\bibitem{jones} {\sc C. K. R. T. Jones} {\it Geometric singular perturbation theory In Dynamical systems}  (1995) Lecture Notes in Math. 1609: 44--118, Springer (Berlin)

\bibitem{LST2} \textsc{J. Llibre, P.R. da Silva and M.A. Teixeira},
\textit{Regularization of discontinuous vector fields via singular perturbation} (2006) J. Dynam. Differential Equations {\bf 19
} 2 309--331

\bibitem{LST} \textsc{J. Llibre, P.R. da Silva and M.A. Teixeira},
\textit{Sliding vector fields via slow--fast systems} (2008) Bulletin of the Belgian Mathematical Society {\bf
15} 851--869

\bibitem{P} {\sc M.M. Peixoto} {\it On structural stability} (1959) Annals of Mathematics, Second Series, {\bf 69} 199--222.

\bibitem{ps} \textsc{T. Poston and I. N. Stewart},
\textit{Catastrophe theory and its applications} (1996) Dover

\bibitem {ST} {\sc J. Sotomayor and M.A. Teixeira},
{\it Regularization of Discontinuous Vector Field} (1998) The 1995 International
Conference on Differential Equations, Lisboa, World Sci. Publ., 207--223

\bibitem {T} {\sc M.A. Teixeira},
{\it Generic bifurcation of sliding vector fields} (1993) J. Math. Anal. Appl., {\bf 176}, 436--457

\end{thebibliography}
\end{document}